\documentclass[12pt]{amsart}

\usepackage{fullpage}
\usepackage{graphicx, cases}
\usepackage{amssymb,verbatim}
\usepackage{epstopdf}
\usepackage{multicol}
\usepackage{ amssymb, amsmath, amsthm, amsfonts,mathrsfs,verbatim}
\usepackage{enumerate}
\usepackage{algorithm}
\usepackage[noend]{algpseudocode}
\usepackage{cleveref}
\usepackage{setdeck}
\usepackage{array}
\usepackage{multirow}

\newtheorem{theorem}{Theorem}[section]
\newtheorem{proposition}[theorem]{Proposition}
\newtheorem{lemma}[theorem]{Lemma}

\newtheorem{corollary}[theorem]{Corollary}

\theoremstyle{definition}
\newtheorem{definition}[theorem]{Definition} 
\newtheorem{example}[theorem]{Example}
 
\newtheorem{remark}[theorem]{Remark}

\newtheorem*{solution*}{Solution}

\theoremstyle{remark}

\title{Metrics on permutations with the same descent set}
\author{Alexander Diaz-Lopez, Kathryn Haymaker, Colin McGarry, Dylan McMahon} 

\begin{document}

\begin{abstract}
Let $S_n$ be the symmetric group on the set $[n]:=\{1,2,\ldots,n\}$. Given a permutation $\sigma=\sigma_1\sigma_2 \cdots \sigma_n \in S_n$, we say it has a descent at index $i$ if $\sigma_i>\sigma_{i+1}$. Let $\mathcal{D}(\sigma)$ be the set of all descents of $\sigma$ and define $\mathcal{D}(S;n)=\{\sigma\in S_n\, | \,\mathcal{D}(\sigma)=S\}$. We study the Hamming metric and $\ell_\infty$-metric on the sets $\mathcal{D}(S;n)$ for all possible nonempty $S\subset[n-1]$ to determine the maximum possible value that these metrics can achieve when restricted to these subsets.
\end{abstract}
\maketitle
\section{Introduction} 

In this article, we find the maximum values that certain permutation metrics can attain when restricted to subsets of permutations that share a given descent set. Let $[n]$ denote the set $\{1, 2, \ldots, n\}$ and $S_n$ be the set of $n!$ symmetries of $[n]$. We write the elements of $S_n$ in one-line notation, that is, for $\sigma\in S_n$  we write $\sigma=\sigma_{1}\sigma_{2}\cdots\sigma_{n}$ to denote the permutation that sends $1\rightarrow \sigma_{1},$ $2\rightarrow \sigma_{2},\ldots, n\rightarrow \sigma_{n}$.
We say $\sigma$ has a \textbf{descent} at position $i$ in $[n-1]$ if $\sigma_{i}>\sigma_{i+1}$.  We define the \textbf{descent set} of $\sigma$, $\mathcal{D}(\sigma)$, as the set of all the indices at which $\sigma$ has a descent. For example, if $\sigma=58327164 \in S_8$ then $\mathcal{D}(\sigma)=\{2,3,5,7\}$.

Given a subset $S\subseteq[n-1]$, let $\mathcal{D}(S;n)$ be the set of permutations in $S_n$ with descent set $S$, that is,
\[\mathcal{D}(S;n)=\{\sigma \in S_n \, | \, \mathcal{D}(\sigma)=S\}.\]
We can partition $S_n$ as a disjoint union of sets of the form $\mathcal{D}(S;n)$ as we range through all possible subsets $S$ of $[n-1]$. The sets $\mathcal{D}(S;n)$ were first studied by MacMahon \cite{macmahon1920}. More recently, Diaz-Lopez et al. \cite{des19} provided some combinatorial results about them and presented some conjectures. This led to a flurry of work related to descent sets \cite{bencs21,gaetz21,jirad23,raychev23}.

Modeling rankings using permutations is a classic application in statistics, and assessing the distance between rankings by studying the distance between pairs of permutations has been considered in this context since at least the 1950s \cite{dh98, kg90, k58}. More recently, permutations have been used to model data encoding structures, specifically error-correcting codes, including for flash memory storage \cite{bm10, b74permutation, kltt10}. In these data representation models, the distance between pairs of permutations is an indicator of the error correction capabilities of the code. Metrics that have been used in this context include the Hamming metric, the $\ell_{\infty}-$ (Chebyshev) metric, the Ulam metric, and the Kendall-Tau metric (see, e.g., \cite{farnoud2013error, jiang2008error, shieh2012computing}).

In this article, we study the Hamming metric, which measures the number of indices at which two permutations differ and the $\ell_{\infty}$-metric, which measures the maximum component-wise difference between two permutations. Motivated by a paper that studies these metrics on sets of permutations with a given peak set \cite{diaz-lopez24metrics}, the main purpose of this article is to find the maximum Hamming and $\ell_{\infty}$-metric when restricting to permutations that share a descent set---subsets of the form $\mathcal{D}(S;n)\subset S_{n}$ for $S\subset [n-1]$. Theorem \ref{thm:Hamming} provides a complete characterization of the maximum Hamming metric on all sets $\mathcal{D}(S;n)$ and Theorems \ref{thm:n-idescents} and \ref{thm:singleton} provide the maximum $\ell_{\infty}$-metric on $\mathcal{D}(S;n)$ for particular cases of $S$. Finding the maximum $\ell_{\infty}$-metric on most sets $\mathcal{D}(S;n)$ is still an open problem.

\section{Preliminary definitions}

In this section, we define the main objects of study of this article. 
Given a set $S$, a \textbf{metric} $d$ on $S$ is a map $d:S\times S \to [0,\infty)$ such that for $\sigma, \rho, \tau$ in $S$ we have
\begin{enumerate}
    \item $d(\sigma,\rho)=0$ if and only if $\sigma=\rho$,
    \item $d(\sigma,\rho)=d(\rho,\sigma)$, and
    \item $d(\sigma,\tau)\leq d(\sigma,\rho)+d(\rho,\tau)$.
\end{enumerate}
In this article, we focus our attention on sets $S$ consisting of permutations. We first define some metrics on $S_n$ and then restrict them to subsets of $S_n$. 
\begin{definition}\label{def:metrics}
Let $d_H$, denoting the \textbf{Hamming metric}, be the map $d_H:S_n \times S_n \to [0,\infty)$ such that $d(\sigma,\rho)$ is the number of indices where $\sigma$ and $\rho$ differ. That is, if $\sigma=\sigma_1\sigma_2\ldots \sigma_n$ and $\rho=\rho_1\rho_2\ldots \rho_n$ then 
$$d_H(\sigma,\rho)=|\{i \, | \, \sigma_i\neq \rho_i\}|.$$
Let $d_\ell$, denoting the  \textbf{$\ell_\infty$-metric}, be the map $d_{\ell}:S_n \times S_n \to [0,\infty)$ such that 
$$d_{\ell}(\sigma,\rho)=\max\{|\sigma_i - \rho_i| \, | \, 1\leq i \leq n\}.$$
\end{definition}
\begin{example}
    Consider the permutations $\sigma = 1 4 7 5 6 8 3 2 $ and $\rho = 1 3 6 2 4 8 7 5$ in $S_8$. 
    They differ in indices $2,3,4,5,7,8$, hence $d_H(\sigma,\rho)=6$. To compute $d_\ell (\sigma, \rho)$ we analyze their component-wise differences to get \[d_\ell(\sigma,\rho) = \max\{|1-1|,|3-4|,|6-7|,|2-5|,|4-6|,|8-8|,|7-3|,|5-2|\} = 4.\]
\end{example}
We now define descent sets.
\begin{definition}
    Given a permutation $\sigma=\sigma_1\sigma_2\ldots \sigma_n$ in $S_n$, its descent set is defined as 
    \[ \mathcal{D}(\sigma)=\{i \in [n-1] \, | \, \sigma_i>\sigma_{i+1}\}.\]
    Given a set $S\subseteq[n-1]$, we let $\mathcal{D}(S;n)$ be the set of permutations in $S_n$ with descent set $S$, that is, 
    \[ \mathcal{D}(S;n)=\{\sigma \in S_n\,|\, \mathcal{D}(\sigma)=S \}.\]
\end{definition}

In this article, we focus on analyzing the Hamming metric and $\ell_\infty$-metric on subsets of the form $ \mathcal{D}(S;n)$ for $S\subset [n-1]$. In particular, we study the maximum of the set
\[d(\mathcal{D}(S;n)):=\{d (\sigma,\rho) \, | \, \sigma,\rho \in \mathcal{D}(S;n) \text{ with } \sigma\neq \rho\},\]
where $d$ is the Hamming metric or the $\ell_\infty$-metric. We take $S$ to be proper and nonempty as $\mathcal{D}(\emptyset;n)=\{1\,2\,\cdots \, n\}$ and $\mathcal{D}([n-1];n)=\{n\,(n-1)\,\cdots \, 1\}$, hence the sets have cardinality one and we cannot compute distances between distinct permutations.

\begin{remark}
A related question is to study the minimum of the set $d(\mathcal{D}(S;n))$, when $d$ is the Hamming metric or the $\ell_\infty$-metric. Let $S\subset [n-1]$ be nonempty and proper. The minimum of the set $d(\mathcal{D}(S;n))$ is always 2 for the Hamming metric and 1 for the $\ell_\infty$ metric, since once you have a permutation $\sigma$ that is not $1\,2\,\,\cdots\,n$ nor $n\,(n-1)\,\,\cdots\,1$, there will always be two consecutive numbers $i$ and $i+1$ that do not appear consecutively in $\sigma$. Swapping $i$ and $i+1$ leads to a different permutation $\sigma'$ with the same descent set as $\sigma$ and with $d_H(\sigma,\sigma')=2$ and $d_\ell(\sigma,\sigma')=1$. 
\end{remark}

\section{Hamming metric on permutations with a given descent set}
Consider  the Hamming metric $d_H$, as presented in Definition \ref{def:metrics}. Given a set $S\subset [n-1]$, we consider the set of permutations $\mathcal{D}(S;n)$. In this section, we show that for almost all sets $S$, we can find a pair of permutations in $\mathcal{D}(S;n)$ that have distance $n$, the maximum possible Hamming distance between any two permutation in $S_n$. Furthermore, if $S$ consists of (only) consecutive descents starting at 1 or consecutive descents ending at $n-1$, we do not achieve the maximum distance (Theorem \ref{thm:Hamming}). 

   For example, consider $S_4$ and let $S\subset[3]$. Table \ref{tab:s-4}  presents the maximum distance of the permutations in $\mathcal{D}(S;4)$. In this case, most of the subsets $S$ consist of consecutive descents at the start or end of the permutation. As $n$ increases, the proportion of sets that meet this criterion goes to 0 as the number of such sets grows linearly in $n$ and the total number of subsets of $[n-1]$ grows exponentially in $n$. 
\begin{table}
    \centering   
\begin{tabular}{|c|c|c|} \hline
    Descent Set $S$ & $\mathcal{D}(S;4)\subset S_4$ & $d_H(\mathcal{D}(S;4))$
    \\\hline
    $\emptyset$ & \{1234\} & ---\\\hline
    \{1\}& \{\textbf{2}134, \textbf{3}124, \textbf{4}123\} & 3\\\hline
    \{2\}&\{1\textbf{3}24, 1\textbf{4}23, 2\textbf{3}14, 2\textbf{4}13, 3\textbf{4}12\} & 4\\\hline
    \{3\}& \{12\textbf{4}3, 13\textbf{4}2, 23\textbf{4}1\}& 3  \\\hline
    \{1,2\}& \{\textbf{32}14, \textbf{42}13, \textbf{43}12\}& 3\\\hline
    \{1,3\}& \{\textbf{2}1\textbf{4}3, \textbf{3}2\textbf{4}1, \textbf{3}1\textbf{4}2, \textbf{4}1\textbf{3}2, \textbf{4}2\textbf{3}1\} & 4\\\hline
    \{2,3\}& \{1\textbf{43}2, 2\textbf{43}1, 3\textbf{42}1\}& 3\\\hline
    \{1,2,3\} & \{\textbf{432}1\} & ---\\\hline
\end{tabular}
 \caption{Table of descent sets and Hamming distances for $S_4$.}
    \label{tab:s-4}
\end{table}
We now present a lemma that is useful in the proof of Theorem \ref{thm:Hamming}.

\begin{lemma}\label{lem:specialentry}
    If $S=\{1,2,\ldots, k\}$ for some $k\in [n-1]$, then any permutation $\sigma=\sigma_1\sigma_2\cdots \sigma_n$ in $\mathcal{D}(S;n)$ must have $\sigma_{k+1}=1$. Similarly, if $S=\{k,k+1,\ldots, n-1\}$ for some $k\in[n-1]$, then $\sigma_k=n$.
\end{lemma}

\begin{proof}
    In the case where $S = \{1,2,\ldots,k\}$ for some $k\in[n-1]$, if $\sigma_i=1$ for $i\leq k$ then $i$ is not a descent as $\sigma_{i+1}>\sigma_i=1$, hence this is not possible. If $\sigma_i=1$ for $i>k+1$ then $i-1$ would be a descent as $\sigma_{i-1}>\sigma_i=1$. This contradicts that $\sigma\in\mathcal{D}(S;n)$. Thus, we must have $\sigma_{k+1}=1$. 

    In the case where $S = \{k, k+1,\ldots, n-1\}$, if $\sigma_i=n$ for $i<k$ then $i$ is a descent as $n=\sigma_i>\sigma_{i+1}$, which contradicts that $\sigma\in\mathcal{D}(S;n)$. If $\sigma_i=n$ for $i>k$, then $i-1$ is not a descent as $\sigma_{i-1}<\sigma_i=n$, which again contradicts $\sigma\in\mathcal{D}(S;n)$. Hence, $\sigma_k=n$. 
\end{proof}

We are ready to prove our main theorem of this section, which consists of a proof by induction that breaks the set of subsets of $[n-1]$ into six cases.
\begin{theorem} \label{thm:Hamming}
Let $n\geq 3$ and $S\subset[n-1]$ be a non-empty proper set, then 
\[ \max(d_H(\mathcal{D}(S;n)))=
\begin{cases}
n-1& \text{ if } S=\{i,i+1,\ldots, j\} \text{ for } i=1 \text{ or } j=n-1\\
n& \text{otherwise}.
\end{cases} \]
\end{theorem}
\begin{proof}
    We proceed by induction on $n$. For $n=3$, when $S=\{1\}$ we get $\mathcal{D}(S;3)=\{ 312, 213\}$, thus $d(\mathcal{D}(\{1\};3))=2$. When $S=\{2\}$, we have $\mathcal{D}(\{2\};3)=\{132, 231\}$, so $d(\mathcal{D}(\{2\};3))=2$. Although not technically needed, the result is shown to be true for $n=4$  in Table~\ref{tab:s-4}. 

Suppose that the result is true for all descent sets $S'\subset [n-2]$ of permutations in $S_{n-1}$ for $n\geq 4$. Let $S\subset [n-1]$ and consider $\mathcal{D}(S;n)$. We proceed by cases depending on the set $S$. Table \ref{tab:induction-steps} summarizes the cases.

\begin{table}
    \centering
    \begin{tabular}{|c|c|c|}\hline
    Case & $S\subset [n-1]$ & $d_H(\mathcal{D}(S;n))$ \\ \hline
    1 & $\{1,2,...,k\} $\text{ with } $k\leq(n-2) $& $n-1$ \\ \hline
    2 & $\{k, (k+1),...,n-2,n-1\}$ \text{ with } $k>1$& $n-1 $\\ \hline
    3 & $\{1,2,...,k,n-1\}$ \text{ with } $k<(n-2)$& $n$ \\ \hline
    4 & $\{k, (k+1),...,(n-2)\}$ \text{ with } $k>1$ & $n$\\ \hline
    5 & \text{A subset of $[n-2]$ except Cases 1 and 4}&  $n$\\ \hline
    6 & \text{$S'\cup\{n-1\}$ where $S'$ is as in Case 5}&  $n$\\ \hline
\end{tabular}
    \caption{Cases to consider in the induction step in  the proof of Theorem~\ref{thm:Hamming}.}
    \label{tab:induction-steps}
\end{table}
\textbf{Case 1:} If $S=\{1,2,\ldots, k\}$ for $1\leq k\leq n-2$, let 
\[\begin{array}{cccccccccccc}
\sigma&=&n& n-1 & n-2 & \cdots & n-(k-2) & n-(k-1) & 1&2 & \cdots&n-k\\
\rho&=&k+1& k& k-1&\cdots &3&2&1&k+2&\cdots&n.   
\end{array}\]
Then, $d_H(\sigma,\rho)=n-1$. By Lemma \ref{lem:specialentry}, this is the largest value that the Hamming metric can obtain in this set, as any permutation with descent set $S$ will have the value of $1$ at index $k+1$. Thus, $d_H(\mathcal{D}(S;n))=n-1$.

\textbf{Case 2:} If $S=\{k,k+1,\ldots, n-2,n-1\}$ for $2\leq k\leq n-1$, let 
\[{\begin{array}{cccccccccccc}
\sigma&=&1& 2 & \cdots & k-2 &k-1 &n &n-1&n-2 & \cdots& k\\
\rho&=&n-(k-1)& n-(k-2)& \cdots &n-2 &n-1&n&n-k&n-(k+1)&\cdots &1.   
\end{array}}\]
Then, $d_H(\sigma,\rho)=n-1$. By Lemma \ref{lem:specialentry}, this is the largest value that the Hamming metric can obtain in this set, as any permutation with descent set $S$ will have the value of $n$ at index $k$. Thus, $d_H(\mathcal{D}(S;n))=n-1$.

\textbf{Case 3}: If $S=\{1,2,\ldots, k,n-1\}$ for $1\leq k\leq n-3$, let 
\[\footnotesize{\begin{array}{ccccccccccccc}
\sigma&=&n& n-1 & \cdots & n-(k-2) & n-(k-1) & 1&2 & \cdots&n-(k+2)&n-k&n-(k+1)\\
\rho&=&n-1& n-2 & \cdots & n-(k-1) & n-k & 2&3 & \cdots&n-(k+1)&n&1
\end{array}}\]
Then, since $\sigma$ and $\rho$ have descent set $S$ and $d_H(\sigma,\rho)=n$, we get $d_H(\mathcal{D}(S;n))=n$.

\textbf{Case 4:} 
If $S=\{k,k+1,\ldots, n-2\}$ for $2\leq k\leq n-2$, let 
\[\footnotesize{\begin{array}{cccccccccccc}
\sigma&=&1& 2 & \cdots &k-1 & n-1 &n-2&n-3 & \cdots&k&n\\
\rho&=&n-k& n-(k-1)& \cdots  &n-2&n&n-(k+1)&n-(k+2)&\cdots&1&n-1. 
\end{array}}\]
Then, since $\sigma$ and $\rho$ have descent set $S$ and $d_H(\sigma,\rho)=n$, we get $d_H(\mathcal{D}(S;n))=n$.

\textbf{Case 5:} Let $S$ be any nonempty, proper subset of $[n-2]$ except those in Cases 1 and 4. By the inductive assumption, $d_H(\mathcal{D}(S;n-1))=n-1$. Hence, there is a pair of permutations $\sigma,\rho \in S_{n-1}$ with descent set $S$ such that $d_H(\sigma,\rho)=n-1$. Since they differ at every index, assume without loss of generality that $\rho_{n-1}\neq n-1$.

Let $\sigma'$ be the permutation $\sigma$ with $n$ appended at the end. Let $\rho''$ be the permutation $\rho$ with $n$ appended at the end, and let $\rho'$ be the permutation $\rho''$ with the values of $n$ and $n-1$ switched. Since $n-1$ and $n$ do not appear consecutively in $\rho''$ then $\rho'$ has descent set $S$. By construction, $\sigma'$ and $\rho'$ have descent set $S$ and differ at every index; hence $d_H(\mathcal{D}(S;n))=n$.

\textbf{Case 6:} Let $S$ be any nonempty, proper subset of $[n-1]$ with $n-1\in S$, except those in Cases 2 and 3. Then $S=S'\cup \{n-1\}$ for some subset $S'\subset [n-2]$, where $S'$ is not of the form of the sets in Cases 1 and 4. There are two subcases. 

First, suppose $n-2\in S'$. Since $S'$ is not one of the sets in Case 4, there must be a minimum $k\in \{3,\ldots, n-2\}$ such that $\{k, k+1, \ldots, n-2\}\subseteq S'$, and $k-1\notin S'$. By the inductive assumption, there are permutations $\sigma, \rho\in \mathcal{D}(S'; n-1)$, with $d(\sigma, \rho)=n-1$. Since all positions of $\sigma$ and $\rho$ are distinct, at most one of $\sigma$ and $\rho$ has their $k^{th}$ index equal to $n-1$. Without loss of generality, suppose that $\rho_k<n-1$. Since $\{k, k+1, \ldots, n-2\}\subseteq S'$, $\rho$ is decreasing from $\rho_k$ to $\rho_{n-1}$. Thus, $\rho_j=n-1$ for some $1\leq j<k-1$.  We form $\sigma', \rho'\in S_{n}$ as follows. 
Let $\sigma'_i=\sigma_i$ for $i=1, \ldots, k$. We set $\sigma'_k=n$, and $\sigma'_{k+i}=\sigma_{k+i-1}$ for $i=1,\ldots, n-k$. Let $\rho'_i=\rho_i$ for $i=[k-1]\setminus \{j\}$. Set $\rho'_j=n$, and $\rho'_k=n-1$. Finally, let $\rho'_{k+i}=\rho_{k+i-1}$ for $i=1,\ldots, n-k$. By construction and by the inductive assumption, we have $\sigma',\rho' \in \mathcal{D}(S;n)$ and $d(\sigma', \rho')=n$. 

Next, suppose $n-2\notin S'$. By the inductive assumption, there are permutations $\sigma, \rho\in \mathcal{D}(S'; n-1)$, with $d(\sigma, \rho)=n-1$. Since all positions of $\sigma$ and $\rho$ are distinct, at most one of $\sigma$ and $\rho$ have their $(n-1)^{th}$ position equal to $n-1$. Without loss of generality, suppose that $\rho_{n-1}\neq n-1$. Since $n-2\notin S'$, it must be the case that $\rho_j=n-1$ for some $j<n-2$.  Form $\sigma', \rho'\in S_n$ as follows. Set $\sigma'_i=\sigma_i$ for $i=1, \ldots, n-2$, $\sigma'_{n-1}=n$, and $\sigma'_{n}=\sigma_{n-1}$. Similarly, set $\rho'_i=\rho_i$ for $i\in[n-2]\setminus\{j\}$. Set $\rho'_j=n$, and $\rho'_{n-1}=n-1$, and $\rho'_{n}=\rho_{n-1}$. By construction and by the inductive assumption, we have $\sigma',\rho' \in \mathcal{D}(S;n)$ and $d(\sigma', \rho')=n$. 

Thus, by the two subcases above, $d_H(\mathcal{D}(S;n))=n$ in this last case.
\end{proof}

\section{$\ell_\infty$-metric on permutations with a given descent set}
We now shift our attention to the $\ell_\infty$-metric and consider the maximum distance that two permutations in $\mathcal{D}(S;n)$ can achieve under the $\ell_\infty$-metric. In Theorem \ref{thm:n-idescents} we consider the case of $S=[n-i]$ for some $i\in\{2,3,\ldots, n-1\}$. In Theorem \ref{thm:singleton} we consider the case $S=\{i\}$.

\begin{theorem}\label{thm:n-idescents}
Fix $n\in \mathbb{N}$ with $n\geq 3$. For any $i \in \{2, \ldots, n-1\}$, let $S_i=[n-i]$, 
then 
\[\max(d_\ell(\mathcal{D}(S_i;n)))=\max\{i-1,n-i\}=
\begin{cases}
    i-1& \text{ if } i\geq \lfloor \frac{n+1}{2}\rfloor\\
    n-i&\text{ if } i\leq \lfloor \frac{n+1}{2}\rfloor.
\end{cases}\]
\end{theorem}
\begin{proof}
We first construct two permutations in $\mathcal{D}(S_i;n)$ that achieve the desired maximum distance and then show that no other pair of permutations can have a larger $\ell_\infty$-distance. Let $\sigma$ and $\rho$ be defined as
 $$\begin{array}{cccccccccc}
 \text{index}&&1&2&\cdots&n-i&n-(i-1)&n-(i-2)&\cdots &n\\\hline
\sigma &=&n-(i-1)&n-i&\cdots&2&1&n-(i-2)&\cdots&n\\
\rho &=&n&n-1&\cdots&i+1& 1&2&\cdots&i
\end{array}$$
It is straightforward to verify that both $\sigma,\rho \in \mathcal{D}(S_i;n)$. To compute their $\ell_{\infty}$-distance, note that 
\[ |\sigma_j-\rho_j|=\begin{cases}
    i-1 & \text{if } j \in \{1,2,\ldots, n-i\}\\
    0 & \text{if } j=n-(i-1)\\
    n-i &\text{if } j \in \{n-(i-2),n-(i-3),\ldots, n\}.
\end{cases}\]
Thus, $d_{\ell}(\sigma,\rho)=\max\{i-1, n-i\}$. 

 For any other pair of permutations $\pi,\tau \in \mathcal{D}(S_i;n)$, we will show $d_\ell(\pi,\tau)\leq \max\{i-1, n-i\}$. Since $\pi$ and $\tau$ must decrease in the first $n-i$ indices, the largest numbers that can appear in the first $n-i$ indices are $n, n-1,\ldots, i+1$, respectively. That is, \[\pi_j,\tau_j\leq n-(j-1)=\rho_j \text{ for } j=1,2,\ldots, n-i.\]  
 Similarly, the smallest entries that can appear in the first $n-i$ indices are $n-(i-1),n-i,\ldots,2$, respectively. Thus,
 \[\pi_j,\tau_j\geq n-(i-1)-(j-1) = \sigma_j \text{ for } j=1,2,\ldots, n-i.\]

Since \[n-(i-1)-(j-1)\leq \pi_j,\tau_j\leq n-(j-1)\text{ for }  j=1,2,\ldots, n-i,\] then 
\[|\pi_j-\tau_j|\leq n-(j-1)-(n-(i-1)-(j-1))=i-1,\]
for $j=1,2,\ldots, n-i$.

We now consider the last $i$ indices. By Lemma \ref{lem:specialentry}, $\pi_{n-(i-1)}=\tau_{n-(i-1)}=1$. Since $\pi$ and $\tau$ must increase in the last $i$ indices,  the largest numbers that can appear in the last $i-1$ indices are $n-(i-2),n-(i-3),\ldots, n$, respectively. That is, \[\pi_j,\tau_j\leq j \text{ for } j=n-(i-2),n-(i-3),\ldots, n.\]  

 Similarly, the smallest entries that can appear in the last $i-1$ indices are $2,3,\ldots, i$, respectively. Thus,
 \[\pi_j,\tau_j\geq j-(n-i)  \text{ for } j=n-(i-2),n-(i-3),\ldots, n.\]
Since \[j\leq \pi_j,\tau_j\leq j-(n-i)\text{ for }  j=n-(i-2),n-(i-3),\ldots, n,\] then
\[|\pi_j-\tau_j|\leq j-(j-(n-i))=n-i,\]
for $j=n-(i-2),n-(i-3),\ldots, n.$ 
Therefore, for any $\pi,\tau \in \mathcal{D}(S_i;n)$, we have $d_\ell(\pi,\tau)\leq \max\{i-1,n-i\}=d_\ell({\sigma,\tau})$.
\end{proof}

We now consider permutations with only one descent. 
\begin{theorem}\label{thm:singleton}

 Let $n\geq 6$ and consider $\mathcal{D}(\{i\};n)$, then
     \[  \max(d_\ell(\mathcal{D}(\{i\};n)))= 
     \begin{cases} n-2 & \text{ for }  i=1, n-1  \\ 
    n-i & \text{ for } i= 2,3, \ldots, \left\lfloor\frac{n}{2}\right\rfloor \\
    i & \text{for } i =\lceil\frac{n}{2}\rceil,\lceil\frac{n}{2}\rceil+1,\ldots, n-2.  
    \end{cases}\]
    
\end{theorem}
\begin{proof}
We will proceed by constructing two permutations in $S_n$ with descent set $\{i\}$ with distance given by the statement and then show that this is the maximum distance any two permutations in $\mathcal{D}(\{i\};n)$ can achieve. Let $\sigma, \rho$ be as follows:
\[\begin{array}{cccccccccccc}
    \sigma&=&1&2&\cdots&i-1&i+1&i&i+2&i+3&\cdots&n \\ 
    \rho&=&n-i+1&n-i+2&\ldots&n-1&n&1&2&3&\ldots&n-i.
\end{array}
    \]
Notice that the difference between the indices of the permutations is given by 
\[|\sigma_j-\rho_j|=\begin{cases}
    n-i&\text{if } j=1,2,\ldots, i-1 \text{ and } i\geq 2\\
    n-(i+1) &\text{if } j=i\\
    i-1 & \text{if } j=i+1\\
    i & \text{if } j=i+2,i+2,\ldots, n \text{ and } i\leq n-2.
\end{cases}\]
A quick check verifies that $d_\ell(\sigma,\rho)$ is given by
\[d_\ell(\sigma,\rho)=\max\{|\sigma_j - \rho_j| \, | \, 1\leq j \leq n\}=\begin{cases} n-2 & \text{ for }  i=1, n-1  \\ 
    n-i & \text{ for } i= 2,3, \ldots, \left\lfloor\frac{n}{2}\right\rfloor \\
    i & \text{for } i =\lceil\frac{n}{2}\rceil,\lceil\frac{n}{2}\rceil+1,\ldots, n-2.  
    \end{cases}\]
We now proceed to show that the maximum distance we can achieve among any pair of distinct permutations in $\mathcal{D}(\{i\}; n)$ is given by $d_\ell(\sigma,\rho)$.

Consider any permutations $\pi,\tau \in \mathcal{D}(\{i\};n)$. Since they are strictly increasing until  index $i$, then the smallest $\pi_1$ and $\tau_1$ can be is 1, and the largest they can be is $n-(i-1)$, as you need to have $i-1$ larger entries for indices $2,3,\ldots, i$. Thus, $1\leq \pi_1,\tau_1\leq n-(i-1)$. By applying a similar argument to entries $2,3,\ldots, i-1$ we get that 
\begin{equation}\label{eq:<i-1}
    j\leq \pi_j,\tau_j\leq n-(i-j) \text{ for } j=1,2,\ldots, i-1.
\end{equation}

For index $i$, we have that $i+1\leq \pi_i,\tau_i$ as the permutations have a descent at $i$, so they need to have indices $1,2,\ldots, i-1, i+1$ all greater than the value at index $i$. Hence, 
\begin{equation}\label{eq:i}
i+1\leq \pi_j,\tau_j\leq n \text{ for } j=i.
\end{equation}

For index $i+1$, the smallest value $\pi_{i+1},\tau_{i+1}$ can attain is $1$ and the largest value they can achieve is $i$ since there are $n-i$ indices that must have larger entries (indices $i, i+2,i+3,\ldots, n$). Hence,
\begin{equation}\label{eq:i+1}
1\leq \pi_j,\tau_j\leq i \text{ for } j=i+1.
\end{equation}

The final $n-i$ entries of $\pi$ and $\tau$ form an increasing sequence. Thus, the smallest value that the indices $i+2, i+3,\ldots, n$ can attain are $2,3,\ldots, n-i$, respectively. Similarly, the largest values they can attain are $i+2, i+3, \cdots, n$, respectively. Thus, 
\begin{equation}\label{eq:>i+1}
    j-i\leq \pi_j,\tau_j\leq j \text{ for } j=i+2,i+3,\ldots, n.
\end{equation}   
Using equations $\eqref{eq:<i-1}-\eqref{eq:>i+1}$, we get that for each $j\in \{1,2,\ldots, n\}$,
\[|\pi_j-\tau_j|\leq|\sigma_j-\rho_j|=\begin{cases}
    n-i&\text{if } j=1,2,\ldots, i-1 \text{ and } i\geq 2\\
    n-(i+1) &\text{if } j=i\\
    i-1 & \text{if } j=i+1\\
    i & \text{if } j=i+2,i+2,\ldots, n-1 \text{ and } i\leq n-2.
\end{cases}.\]
Hence, $\max(d_\ell(\mathcal{D}(\{i\};n)))=d_\ell(\sigma,\rho)$ as desired. 
\end{proof}
\section{Complements of Descent Sets}
In this section, we use a bijection map of $S_n$ to show that $\max(d_\ell(\mathcal{D}(S;n)))$ is equal to $\max(d_\ell(\mathcal{D}(\bar{S};n)))$, where $\bar{S}$ is the complement of $S$ in $[n-1]$. 

\begin{definition}
    For a set $S\subseteq [n-1]$, define the \textbf{complement of $S$} in $[n-1]$ as $\bar{S}=[n-1]\setminus S$. 
\end{definition}

Consider the map $\Phi:S_n\to S_n$ where for $\sigma=\sigma_1\sigma_2\cdots \sigma_n \in S_n$ we have \[\Phi(\sigma)= (n+1-\sigma_1)\,(n+1-\sigma_2)\, \ldots\, (n+1-\sigma_n).\] Geometrically, graphing the points $(i,\sigma_i)$ in $\mathbb{R}^2$, the map $\Phi$ is flipping the graph of the permutation across the horizontal line given by $y=(n+1)/2$.

The next two propositions show that the map $\Phi$ sends permutations in $\mathcal{D}(S;n)$ to permutations in $\mathcal{D}(\bar{S};n)$ in a bijective manner.
\begin{proposition}
    The map $\Phi:S_n\to S_n$ is a bijection. Furthermore, it is its own inverse, that is, $\Phi\circ\Phi$ is the identity map on $S_n$.
\end{proposition}
\begin{proof}
The first statement follows from the second. To prove the second, note that for $\sigma=\sigma_1\sigma_2\cdots \sigma_n \in S_n$ we have that for each $i \in [n]$,
\[(\Phi(\Phi(\sigma)))_i= n+1-(n+1-\sigma_i)=\sigma_i,\]
hence $\Phi(\Phi(\sigma))=\sigma$.
\end{proof}

\begin{proposition}\label{prop:sent}
If $\sigma\in \mathcal{D}(S;n)$, then $\Phi(\sigma)\in \mathcal{D}(\bar{S};n)$.
\end{proposition}
\begin{proof}
    If $\sigma$ has a descent at index $i$, then $\sigma_i > \sigma_{i+1} $, which implies $-\sigma_i < -\sigma_{i+1}$. Adding $n+1$ to both sides, we get $ n + 1 - \sigma_i < n + 1 - \sigma_{i+1}$. Therefore, $\Phi(\sigma_i) < \Phi(\sigma_{i+1})$, so every index that is a descent in $\sigma$ will no longer be one in $\Phi(\sigma)$.

    Now, let $i\not\in S$, so $\sigma$ does not have a descent at index $i$. Then $\sigma_i < \sigma_{i+1}$. This implies that $-\sigma_i > -\sigma_{i+1}$. Adding $n+1$ to both sides, we get $n +1 - \sigma_i > n + 1 - \sigma_{i+1}$. Therefore,  $\Phi(\sigma_i) > \Phi(\sigma_{i+1})$. Every index that is not a descent in $\sigma$ is now a descent in $\Phi(\sigma)$. Therefore, if $\sigma\in \mathcal{D}(S;n)$, then $\Phi(\sigma)\in \mathcal{D}(\bar{S};n)$. 
\end{proof}

 The next result is needed to prove Theorem \ref{thm:complement}. It states that the $\ell_\infty$-metric is invariant under the map $\Phi$.
\begin{lemma}\label{lem:preserve}
For all $\sigma, \rho\in S_n$,
    $d_{\ell}(\sigma, \rho)=d_{\ell}(\Phi(\sigma), \Phi(\rho))$.
\end{lemma}
\begin{proof}
Given $\sigma,\rho \in S_n$ and $i\in[n]$, note that
\[|\sigma_i-\rho_i|=|(n+1-\sigma_i)-(n+1-\rho_i)|=|\Phi(\sigma)_i-\Phi(\rho)_i|.\]
Since when calculating $d_\ell$, we are only concerned with the absolute value of the difference between indices we get that $d_{\ell}(\sigma, \rho)=d_{\ell}(\Phi(\sigma), \Phi(\rho))$.
\end{proof}

\begin{theorem}\label{thm:complement}
    Let $S$ be any proper, nonempty subset of $[n]$. Then \[ \max (d_\ell(\mathcal{D}(S;n)))=\max(d_\ell (\mathcal{D}(\bar{S};n)).\]
\end{theorem} 

\begin{proof}
Let $m=\max (d_\ell(\mathcal{D}(S;n)))$ and $\sigma,\rho \in \mathcal{D}(S;n)$ such that $d_\ell(\sigma,\rho)=m$. By Proposition \ref{prop:sent}, $\Phi(\sigma),\Phi(\rho)\in\mathcal{D}(\bar{S};n)$. By Lemma \ref{lem:preserve}, $d_\ell(\Phi(\sigma),\Phi(\rho))=m$, hence $\max(d_\ell (\mathcal{D}(\bar{S};n))\leq \max(d_\ell (\mathcal{D}(S;n))$. Applying the same argument to $\mathcal{D}(\bar{S};n)$, we get $\max(d_\ell (\mathcal{D}(S;n))\leq \max(d_\ell (\mathcal{D}(\bar{S};n))$.
\end{proof}

We can apply Theorem \ref{thm:complement} to the results in Theorems \ref{thm:n-idescents} and \ref{thm:singleton} to get the following two corollaries.

\begin{corollary} Let $n\geq 3$. 
    For any $i \in \{2, \ldots, n-1]$, let $\bar{S_i}=\{n-i+1, n-i+2,\ldots, n-1\}$ and let 
\[ \mathcal{D}(\bar{S_i};n)=\{\sigma \in S_n\, | \, \mathcal{D}(\sigma)=\bar{S_i}\},\]
then  
\[\max(d_\ell(\mathcal{D}(\bar{S_i};n)))=\max\{i-1,n-i\}=
\begin{cases}
    i-1& \text{ if } i\geq \lfloor \frac{n+1}{2}\rfloor\\
    n-i&\text{ if } i\leq \lfloor \frac{n+1}{2}\rfloor.
\end{cases}\]
\end{corollary}

\begin{corollary}
 Let $n\geq 6$ and consider $\mathcal{D}(\bar{\{i\}};n)$, 
 then
    \[  \max(d_\ell(\mathcal{D}(\bar{\{i\}};n)))= 
     \begin{cases} n-2 & \text{ for }  i=1, n-1  \\ 
    n-i & \text{ for } i= 2,3, \ldots, \left\lfloor\frac{n}{2}\right\rfloor \\
    i & \text{for } i =\lceil\frac{n}{2}\rceil,\lceil\frac{n}{2}\rceil+1,\ldots, n-2.  
    \end{cases}\]
\end{corollary}
\section{Acknowledgments}
The authors thank Villanova University’s Co-MaStER program.  
A. Diaz-Lopez’s research is supported in part by the National Science Foundation grant DMS-2211379.

\bibliographystyle{amsplain}
\bibliography{MetricsDescentSet}
\end{document}